\documentclass[11pt]{amsart}
\usepackage[usenames]{color}
\usepackage{mathrsfs}
\usepackage[hidelinks]{hyperref}
\usepackage{amssymb,amsmath,amsfonts,latexsym,amsthm}
\usepackage{graphicx}
\usepackage{float}	

\usepackage{orcidlink}

\usepackage{tikz}
\usetikzlibrary{calc}

\calclayout

\newcommand{\seqnum}[1]{\href{http://oeis.org/#1}{\underline{#1}}}

\newcount\tilpos

\newcommand{\Tiling}[2][4]{
\begin{tikzpicture}[x=0.7cm,y=0.7cm,baseline={(0,0.5)}, line join=round]
  \def\N{#1}

  \tikzset{
    board/.style={draw, line width=1.0pt},
    Astyle/.style={draw, line width=0.9pt}, 
    Bstyle/.style={draw, line width=0.9pt},  
    Cstyle/.style={draw, line width=0.9pt},   
    lab/.style={font=\small}
  }

  \def\LenA{1}\def\LenB{1}\def\LenC{2}

  \def\TileA##1{%
    \draw[Astyle] (##1,0) rectangle ++(1,1);
    \node[lab] at ({##1+0.5},0.5) {$at$};
  }
  \def\TileB##1{%
    \draw[Bstyle] (##1,0) rectangle ++(1,1);
    \node[lab] at ({##1+0.5},0.5) {$b$};
  }
  \def\TileC##1{%
    \draw[Cstyle] (##1,0) rectangle ++(2,1);
    \node[lab] at ({##1+1.0},0.5) {$c$};
  }

  \tilpos=0
  \foreach \t in {#2} {%
    \edef\x{\number\tilpos}%
    \csname Tile\t\endcsname{\x}%
    \global\advance\tilpos by \csname Len\t\endcsname\relax
  }

  \draw[board] (0,0) rectangle (\N,1);
\end{tikzpicture}%
}

\newcommand{\SetAColor}[2]{%
  \expandafter\def\csname Acolor#1\endcsname{#2}%
}

\newcount\tilpos

\newcommand{\TilinC}[2][4]{
\begin{tikzpicture}[x=0.7cm,y=0.7cm,baseline={(0,0.5)}, line join=round]
  \def\N{#1}

  \tikzset{
    board/.style={draw, line width=1.0pt},
    Bstyle/.style={draw, line width=0.9pt, fill=white},
    Cstyle/.style={draw, line width=0.9pt, fill=white},
    lab/.style={font=\small}
  }

  \def\LenA{1}\def\LenB{1}\def\LenC{2}

  \def\sp{ }%
  \def\Trimmed{}%
  \def\TrimLead##1{\expandafter\TrimLeadAux##1\relax}%
  \def\TrimLeadAux##1##2\relax{%
    \ifx##1\sp
      \expandafter\TrimLeadAux##2\relax
    \else
      \def\Trimmed{##1##2}%
    \fi
  }%

  \def\SplitFirst##1##2\relax{\def\TileType{##1}\def\TileRest{##2}}%
  \def\LBrack{[}%
  \def\GetFirst##1##2\relax{\def\RestFirst{##1}\def\RestTail{##2}}%
  \def\StripBrackets[##1]\relax{\def\Aindex{##1}}%

  \def\TA{A}\def\TB{B}\def\TC{C}

  \tilpos=0
  \foreach \rawtile in {#2} {%
    \TrimLead{\rawtile}%
    \edef\x{\number\tilpos}%
    \expandafter\SplitFirst\Trimmed\relax

    \ifx\TileType\TA
      \def\Aindex{1}%
      \ifx\TileRest\empty\else
        \expandafter\GetFirst\TileRest\relax
        \ifx\RestFirst\LBrack
          \expandafter\StripBrackets\TileRest\relax
        \else
          \def\Aindex{\TileRest}%
        \fi
      \fi
      \ifcsname Acolor\Aindex\endcsname
        \edef\Afill{\csname Acolor\Aindex\endcsname}%
      \else
        \def\Afill{gray!20}%
      \fi
      \draw[draw, line width=0.9pt, fill=\Afill] (\x,0) rectangle ++(1,1);
      \node[lab] at ({\x+0.5},0.5) {$a$};

    \else\ifx\TileType\TB
      \draw[Bstyle] (\x,0) rectangle ++(1,1);
      \node[lab] at ({\x+0.5},0.5) {$b$};

    \else\ifx\TileType\TC
      \draw[Cstyle] (\x,0) rectangle ++(2,1);
      \node[lab] at ({\x+1.0},0.5) {$c$};
    \fi\fi\fi

    \global\advance\tilpos by \csname Len\TileType\endcsname\relax
  }

  \draw[board] (0,0) rectangle (\N,1);
\end{tikzpicture}%
}

\SetAColor{1}{green!20}
\SetAColor{2}{yellow!20}
\SetAColor{3}{orange!25}

\makeatother

\newcommand{\N}{{\mathbb N}}

\newcommand{\wt}{{\texttt{wt}}}

\newcommand{\Ft}[1]{\mathscr{F}_{#1}}
\newcommand{\Lt}[1]{\mathscr{L}_{#1}}

\theoremstyle{plain}
\numberwithin{equation}{section}
\newtheorem{thm}{Theorem}[section]
\newtheorem{corollary}{Corollary}[section]
\newtheorem{theorem}[thm]{Theorem}

\newtheorem{proposition}[thm]{Proposition}

\newtheorem{remark}{Remark}

\title[Holonomic Sequences from Laplace Transforms of  Second Order Sequences]{Holonomic Sequences Arising from Laplace Transforms of Generalized Second Order Sequences}


\date{\today}

\author[R. Fl\'orez]{Rigoberto Fl\'orez \orcidlink{0000-0002-3644-9358}}
\address{Department of Mathematics and Data Analytics\\
		The Citadel\\
		Charleston, SC \\
		U.S.A.}
  \email{rflorez1@citadel.edu}
\author[R.  Higuita-Diaz]{Robinson Higuita-Diaz \orcidlink{0000-0003-3305-0381}}
 \address{Escuela de Matem\'aticas y Estad\'istica, Universidad Pedag\'ogica y Tecnol\'ogica de Colombia, Tunja, Colombia.}
\email{ robinson.higuita@uptc.edu.co}
\author[J. L. Ram\'{\i}rez]{Jos\'e L. Ram\'{\i}rez \orcidlink{0000-0002-8028-9312}}

\address{Departamento de Matem\'aticas,  Universidad Nacional de Colombia,  Bogot\'a, Colombia.}
\email{jlramirezr@unal.edu.co}

\begin{document}

\begin{abstract}
We study the FiboLaplace and LucasLaplace sequences obtained by applying
the Laplace transform to generalized Fibonacci-type and Lucas-type
polynomials. We relate these families and show that, for each fixed
transform parameter, the FiboLaplace sequence is holonomic. We give a
combinatorial interpretation in terms of weighted colored tilings,
leading to recurrence relations, identities, and a triangular
refinement by the number of dominoes. We also derive two continued-fraction expansions and  obtain a second combinatorial interpretation in terms of 
generalized Motzkin paths. These results extend earlier work of Givens and Moll.
\end{abstract}

\maketitle
\begin{center}
\emph{Dedicated to Professor Víctor H. Moll.\footnote{In recognition of
his forty years at Tulane University and in deep appreciation of his
invaluable support of Colombian mathematics and the many opportunities
he has created for Colombian students.}}
\end{center}

\section{Introduction}

Sequences defined by second-order recurrences arise naturally in number theory and combinatorics. Among the best-known examples are the Fibonacci and Lucas sequences and their polynomial analogues. In this paper, we consider the following generalized Fibonacci-type and Lucas-type polynomial sequences.

Let $a,b,c\in\mathbb{Q}$ with $a\neq 0$ and $c\neq 0$. The \emph{generalized Fibonacci-type polynomial sequence}
$\bigl(\Ft{n}(t)\bigr)_{n\geq 0}$ is defined by $\Ft{0}(t)=0$, $\Ft{1}(t)=1$, and
\begin{equation}\label{Fibonacci;general:FT}
\Ft{n}(t)=(at+b)\Ft{n-1}(t)+c\,\Ft{n-2}(t), \qquad n\geq 2.
\end{equation}
The corresponding \emph{generalized Lucas-type polynomial sequence}
$\bigl(\Lt{n}(t)\bigr)_{n\geq 0}$ is defined by
$\Lt{0}(t)=2$, $\Lt{1}(t)=at+b$, and
\begin{equation}\label{Fibonacci;general:LT}
\Lt{n}(t)
=
(at+b)\Lt{n-1}(t)+c\,\Lt{n-2}(t),
\qquad n\geq 2.
\end{equation}
More general initial conditions may be used for Lucas-type polynomials; see, for example,
\cite{florezHiguitaMuk2018}. We use the normalization above throughout.

By choosing suitable values of $a$, $b$, and $c$, one recovers several classical polynomial families. For example, the choice $a=1$, $b=0$, and $c=1$ gives the classical Fibonacci and Lucas polynomial sequences. Evaluating these polynomials at $t=1$ yields the corresponding numerical sequences, namely the Fibonacci numbers $F_n$ and the Lucas numbers $L_n$, respectively. Table~\ref{familiarfibonacci} gives specific instances of these choices. Further properties of these polynomial sequences may be found in \cite{AmdeberhanChenMollSagan, FlorezJC, florezHiguitaMuk2018} and the references therein.

\begin{table} [!ht]
\begin{center}\scalebox{0.8}{
\begin{tabular}{|l|c|c|l|} \hline
  Polynomial            &   $n=0$    &  $n=1$	& Recurrence for $n\geq 2$ 						       \\	
    			  \hline   \hline
  Fibonacci             	 & $0$	    &$1$      	&$F_{n}(t) = t F_{n - 1}(t) + F_{n - 2}(t)$	 	       \\
  Lucas 	             	 &$2$	    & $t$ 	 	&$L_n(t)= t L_{n - 1}(t) + L_{n - 2}(t)$                \\ 						
  Pell			    		 &$0$	    & $1$       &$P_n(t)= 2t P_{n - 1}(t) + P_{n - 2}(t)$               \\
  Pell-Lucas 	    		 &$2$	    &$2t$       &$Q_n(t)= 2t Q_{n - 1}(t) + Q_{n - 2}(t)$               \\
  Fermat  	                 &$0$	    & $1$      	&$\Phi_n(t)= 3t\Phi_{n-1}(t)-2\Phi_{n-2}(t) $           \\
  Fermat-Lucas	             &$2$	    &$3t$  		&$\vartheta_n(t)=3t\vartheta_{n-1}(t)-2\vartheta_{n-2}(t)$\\
  Morgan-Voyce	(first type)              &$0$		&$1$      	&$B_n(t)= (t+2) B_{n-1}(t)-B_{n-2}(t) $  	 	         \\
  Morgan-Voyce (second type) 	             &$2$		&$t+2$      &$C_n(t)= (t+2) C_{n-1}(t)-C_{n-2}(t)$  	 	         \\
  Vieta 		             &$0$ 	   	&$1$	    &$V_n(t)=t V_{n-1}(t)-V_{n-2}(t)$ 	    \\
  Vieta-Lucas 		         &$2$ 	   	&$t$	    &$v_n(t)=t v_{n-1}(t)-v_{n-2}(t)$      \\
    \hline
\end{tabular}}
\end{center}
\caption{Some classical Fibonacci-type and Lucas-type polynomial families.}\label{familiarfibonacci}
\end{table}

The work of Givens and Moll~\cite{Givens} provides the starting point for the present paper. They studied arithmetic properties of the sequence \[ e_n=\int_0^\infty e^{-t}F_n(t)\,dt, \] where $F_n(t)$ denotes the $n$th Fibonacci polynomial, with particular emphasis on its $p$-adic valuation. The sequence $(e_n)_{n\geq 1}$ is listed in the OEIS as \cite[\seqnum{A362787}]{sloane}. This integral can be viewed as a special value of the Laplace transform. Indeed, for a polynomial $f(t)$ and $s>0$, the Laplace transform is given by 
\[ \mathcal{L}\{f(t)\}(s) = \int_0^\infty e^{-st}f(t)\,dt. \] Thus, the sequence studied by Givens and Moll is given by $e_n=\mathcal{L}\{F_n(t)\}(1)$.

Motivated by this point of view, we extend the construction of Givens and Moll to the generalized Fibonacci-type and Lucas-type polynomial families defined above. Since $\Ft{n}(t)$ and $\Lt{n}(t)$ are polynomials in $t$, their Laplace transforms are well defined for every $s>0$. We therefore define, for each $s>0$, the \emph{FiboLaplace sequence} $\bigl(e_n(s)\bigr)_{n\geq 0}$ by

\[
e_n(s):=\mathcal{L}\{\Ft{n}(t)\}(s)
=\int_0^\infty e^{-st}\Ft{n}(t)\,dt,
\]
and the \emph{LucasLaplace sequence} $\bigl(l_n(s)\bigr)_{n\geq 0}$ by
\[
l_n(s):=\mathcal{L}\{\Lt{n}(t)\}(s)
=\int_0^\infty e^{-st}\Lt{n}(t)\,dt.
\]

For each fixed $s>0$, we prove that the sequence $(e_n(s))_{n\geq 0}$ is holonomic; that is, it satisfies a linear recurrence whose coefficients are polynomials in $n$. We also establish relations between the FiboLaplace and LucasLaplace sequences, derive a first-order differential equation for the ordinary generating function of $(e_n(s))_{n\geq 0}$, and obtain formulas for the derivatives of $e_n(s)$ with respect to $s$. 

We then give a combinatorial interpretation of the FiboLaplace sequence. Although the model extends to arbitrary values of $s$, we focus on the case $s=1$, where the weights take a particularly simple form. This model yields explicit formulas, recurrence relations, and a continued-fraction expansion for the ordinary generating function. A second continued-fraction representation leads to an interpretation in terms of generalized Motzkin paths.

The paper is organized as follows. Section~2 develops recurrence and
generating-function properties of the FiboLaplace and LucasLaplace
sequences. Section~3 gives a combinatorial interpretation of the
FiboLaplace sequence in terms of weighted colored tilings, together
with explicit formulas, identities, and a triangular refinement by the
number of dominoes. Section~4 derives two continued-fraction expansions and 
gives an interpretation in terms of weighted generalized Motzkin paths.

\section{Some Properties of the FiboLaplace and LucasLaplace Sequences}

In this section, we establish several basic properties of the FiboLaplace and LucasLaplace sequences. When no confusion can arise, we suppress the dependence on $s$ and write $e_n$ and $l_n$ instead of $e_n(s)$ and $l_n(s)$, respectively. For each fixed $n\geq 0$, the quantities $e_n(s)$ and $l_n(s)$ are rational functions of $s$. As an illustration of the construction, Table~\ref{tab:example-FiboLaplace} lists the first few Fibonacci polynomials $F_n(t)$ and their Laplace transforms.

\begin{table}[ht!]
\centering
\renewcommand{\arraystretch}{1.25}
\setlength{\tabcolsep}{6pt}
\begin{tabular}{c|c|c}
\hline
$n$ & $F_n(t)$ & $e_n(s)=\mathcal{L}\{F_n(t)\}(s)$ \\
\hline\hline
0 & $0$ & $0$ \\
1 & $1$ & $\frac{1}{s}$ \\
2 & $t$ & $\frac{1}{s^{2}}$ \\
3 & $t^{2}+1$ & $\frac{s^{2}+2}{s^{3}}$ \\
4 & $t^{3}+2t$ & $\frac{2(s^{2}+3)}{s^{4}}$ \\
5 & $t^{4}+3t^{2}+1$ & $\frac{s^{4}+6s^{2}+24}{s^{5}}$ \\
6 & $t^{5}+4t^{3}+3t$ & $\frac{3(s^{4}+8s^{2}+40)}{s^{6}}$ \\
7 & $t^{6}+5t^{4}+6t^{2}+1$ & $\frac{s^{6}+12s^{4}+120s^{2}+720}{s^{7}}$ \\
\hline
\end{tabular}
\caption{The first Fibonacci polynomials $F_n(t)$ and their Laplace transforms $e_n(s)$.}
\label{tab:example-FiboLaplace}
\end{table}

We will use the following identities relating the Fibonacci-type and Lucas-type polynomial sequences:
\begin{align}
\Lt{n}(t)
&=
c\,\Ft{n-1}(t)+\Ft{n+1}(t),
&& n\geq 1,
\label{IdentidadLucaFibonacci}
\\
\Lt{n}'(t)
&=
na\,\Ft{n}(t),
&& n\geq 0.
\label{IdentidadRelaciónDerivadas}
\end{align}
The first identity appears in
\cite[Identity~2]{FlorezMcAnallyMuk}, and the second in
\cite{florezHiguitaRamirez}.

The next theorem uses these identities to relate the FiboLaplace and LucasLaplace sequences directly. In particular, it expresses $l_n(s)$ in terms of $e_n(s)$.

\begin{theorem}\label{Relacion:e_nyl_n}
For every $n\ge 0$
\[
s\,l_n(s)=na\,e_n(s)+\Lt{n}(0).
\]
\end{theorem}

\begin{proof}
Applying the Laplace transform to
\eqref{IdentidadRelaciónDerivadas} gives
\[
\mathcal{L}\{\Lt{n}'(t)\}(s)
=
na\,\mathcal{L}\{\Ft{n}(t)\}(s)
=
na\,e_n(s).
\]
On the other hand,  the Laplace transform satisfies 
\[\mathcal{L}\{f'(t)\}(s)=s\,\mathcal{L}\{f(t)\}(s)-f(0),\]
which follows from integration by parts. Applying this equality  with $f(t)=\Lt{n}(t)$ yields 
\[
\mathcal{L}\{\Lt{n}'(t)\}(s)=s\,\mathcal{L}\{\Lt{n}(t)\}(s)-\Lt{n}(0)=s\,l_n(s)-\Lt{n}(0).
\]
Equating these two expressions gives
\[
s\,l_n(s)-\Lt{n}(0)=na\,e_n(s),
\]
which proves the result.
\end{proof}

The next result gives recurrence relations for the FiboLaplace and LucasLaplace sequences.

\begin{theorem}\label{en-recursiva}
Let $a,b,c\in\mathbb{Q}$ with $a\neq 0$ and $c\neq 0$.
\begin{enumerate}
\item\label{en-recursivaPart1}
For $n\geq 1$, the FiboLaplace sequence satisfies
\[
s\,e_{n+1}(s)=na\,e_n(s)-c\,s\,e_{n-1}(s)+\Lt{n}(0),
\]
with initial values $e_0(s)=0$ and $e_1(s)=1/s$.

\item\label{en-recursivaPart2}
For $n\geq 2$, the LucasLaplace sequence satisfies
\[
s(n-1)l_{n+1}(s)
= a(n^{2}-1)l_n(s)
-(n+1)c\bigl(sl_{n-1}(s)-\Lt{n-1}(0)\bigr)
+(n-1)\Lt{n+1}(0),
\]
with initial values $$l_0(s)=2/s, \quad l_1(s)=(a+bs)/s^2,\; \text{ and } \; l_2(s)=1/s^3((b^2+2c)s^2+2ab\,s+2a^2).$$
\end{enumerate}
\end{theorem}

\begin{proof}
Applying the Laplace transform to
\eqref{IdentidadLucaFibonacci} gives
\begin{equation}\label{eq:Laplace-LF2}
l_n(s)=c\,e_{n-1}(s)+e_{n+1}(s), \qquad n\geq 1.
\end{equation}
Multiplying \eqref{eq:Laplace-LF2} by $s$ and using Theorem~\ref{Relacion:e_nyl_n}, we obtain
\[
na\,e_n(s)+\Lt{n}(0)=c\,s\,e_{n-1}(s)+s\,e_{n+1}(s),
\]
which proves the first part.

For the second part, Theorem~\ref{Relacion:e_nyl_n} gives
\begin{equation*}
e_k(s)=\frac{s}{ka}\,l_k(s)-\frac{\Lt{k}(0)}{ka}, \qquad k\geq 1.
\end{equation*}
Substituting this expression with $k=n-1$ and $k=n+1$ into \eqref{eq:Laplace-LF2}, and multiplying by $a(n^2-1)$, gives
\[
a(n^2-1)\,l_n(s)
= s\Bigl((n-1)l_{n+1}(s)+(n+1)c\,l_{n-1}(s)\Bigr)
-\Bigl((n-1)\Lt{n+1}(0)+(n+1)c\,\Lt{n-1}(0)\Bigr).
\]
Solving for $l_{n+1}(s)$ gives the stated recurrence.  The stated initial values follow directly from the definitions.
\end{proof}

The next result eliminates the dependence on the term $\Lt{n}(0)$ appearing in Theorem~\ref{en-recursiva}. As a consequence, we obtain a homogeneous recurrence relation for the FiboLaplace sequence. However, the resulting recurrence has order four, rather than order two.

\begin{theorem}\label{thm:en-4term}
For $n\ge 3$,
\[
s\,e_{n+1}(s)
=(bs+na)\,e_n(s)
-b(n-1)a\,e_{n-1}(s)
+\bigl(bcs-c(n-2)a\bigr)e_{n-2}(s)
+c^{2}s\,e_{n-3}(s),
\]
with the initial values
\[
e_0(s)=0,\qquad e_1(s)=\frac{1}{s},\qquad e_2(s)=\frac{a+bs}{s^2},\qquad
e_3(s)=\frac{(b^2+c)s^2+2ab\,s+2a^2}{s^3}.
\]
\end{theorem}
\begin{proof}
By Part~\ref{en-recursivaPart1} of Theorem~\ref{en-recursiva}, 
\begin{equation}\label{eq:defAn}
s\,e_{n+1}(s)-na\,e_n(s)+c\,s\,e_{n-1}(s)=\Lt{n}(0), \qquad n\geq 1.
\end{equation}
Define $A_n(s):=s e_{n+1}(s)-na\,e_n(s)+cs\,e_{n-1}(s)$. Then \eqref{eq:defAn} gives $A_n(s)=\Lt{n}(0)$. Evaluating the Lucas-type recurrence \eqref{Fibonacci;general:LT} at $t=0$ gives, 
\[
\Lt{n}(0)=b\,\Lt{n-1}(0)+c\,\Lt{n-2}(0), \qquad n\geq 2.
\]
It follows that
\[
A_n(s)=b\,A_{n-1}(s)+c\,A_{n-2}(s)\qquad n\geq 3.
\]
Substituting the definitions of $A_n(s)$, $A_{n-1}(s)$, and $A_{n-2}(s)$ and collecting terms gives the stated recurrence.
\end{proof}

For each fixed $s>0$, Theorem~\ref{thm:en-4term} shows that
$\bigl(e_n(s)\bigr)_{n\geq 0}$ is P-recursive, or equivalently holonomic,
since it satisfies a linear recurrence with coefficients polynomial in
$n$. Consequently, its ordinary generating function is D-finite, that is  it
satisfies a linear differential equation with polynomial coefficients.
See, for example, \cite{Kauers}.

We now turn to the ordinary generating function of the FiboLaplace sequence. Define
\[
E(t;s):=\sum_{n\geq 0}e_n(s)t^n.
\]

\begin{theorem}\label{thm:ODE-E}
The generating function $E(t;s)$ satisfies the first-order linear differential equation
\[
a t^2 E'(t;s)-s(1+c t^2)E(t;s)
+\frac{t(1+c t^2)}{1-bt-ct^2}=0,
\]
with initial condition $E(0;s)=0$.
\end{theorem}
\begin{proof}
Let
\[
L(t):=\sum_{n\ge 0}\Lt{n}(0)t^n.
\] Evaluating the Lucas-type recurrence at zero gives
\[
\Lt{n}(0)=b\,\Lt{n-1}(0)+c\,\Lt{n-2}(0),
\qquad n\geq 2,
\]
with initial values $\Lt{0}(0)=2$ and $\Lt{1}(0)=b$. Therefore,
\[
L(t)=\frac{2-bt}{1-bt-ct^2}.
\]
Multiplying the first recurrence of Theorem~\ref{en-recursiva} by $t^{n+1}$
and summing over $n\geq 1$, we obtain
\[
s\bigl(E(t;s)-e_0-e_1t\bigr)=a t^2 E'(t;s)-c s t^2 E(t;s)+t\bigl(L(t)-\Lt{0}(0)\bigr).
\]
Using the initial values $e_0(s)=0$, $e_1(s)=1/s$, and $\Lt{0}(0)=2$, and substituting the expression for $L(t)$ gives the desired differential equation. The initial condition follows immediately from the definition of $E(t;s)$.
\end{proof}

The appearance of a first-order linear differential equation for the ordinary generating function is consistent
with a general feature of linear recurrences whose coefficients depend polynomially in $n$. A prototypical example is the three-term recurrence
\[
a_{n+1}+a_{n-1}=(\alpha n+\beta)a_n\qquad(\alpha\ne 0),
\]
with $a_0=0$ and $a_1=1$, studied by Janson~\cite{JansonDivergentGF}. In that case, the ordinary generating function also satisfies a first-order linear differential equation. Although such generating functions may have radius of convergence zero, they remain well defined as formal power series and may still be studied by analytic summation methods.

\subsection{Derivatives of the FiboLaplace sequence}

For a polynomial $f(t)$ and $s>0$, the standard differentiation rule for the Laplace transform gives
\begin{equation}\label{ideder}
\mathcal{L}\{t^n f(t)\}(s)
=
(-1)^n\frac{d^n}{ds^n}\mathcal{L}\{f(t)\}(s),
\qquad n\geq 0.
\end{equation}

The following result expresses the derivatives of $e_m(s)$ in terms of shifted entries of the FiboLaplace sequence.

\begin{theorem}\label{teor:Derivativ}
Let $m\geq n \geq 0$, and let $s>0$. Then
\[
\frac{d^{n}e_{m}(s)}{ds^{n}}
=\frac{1}{a^n}\sum_{\substack{j,k,\ell\ge 0\\ j+k+\ell=n}}
\binom{n}{j,k,\ell}(-1)^{j}\, b^{k} c^{\ell}\, e_{\,m+j-\ell}(s),
\]
where $\binom{n}{j,k,\ell}=\frac{n!}{j!\,k!\,\ell!}$.
\end{theorem}

\begin{proof}
The case $n=0$ is immediate. Suppose that $n\geq 1$. Since $m\geq n$, we have $m\geq 1$, and the defining recurrence gives
\[
t\,\Ft{m}(t)
=
\frac{1}{a}
\Bigl(
\Ft{m+1}(t)-b\,\Ft{m}(t)-c\,\Ft{m-1}(t)
\Bigr).
\]
Taking Laplace transforms and using  \eqref{ideder} with $n=1$, we obtain 
\begin{equation}\label{eq:first-deriv-proof}
-\frac{d}{ds}e_m(s)=\frac1a\Bigl(e_{m+1}(s)-b\,e_m(s)-c\,e_{m-1}(s)\Bigr).
\end{equation}
Define the linear operator
\[
(\mathcal D f)_m:=f_{m+1}-b f_m-c f_{m-1}.
\]
Then \eqref{eq:first-deriv-proof} becomes
\[
-\frac{d}{ds}e_m(s)=\frac1a(\mathcal D e)_m.
\]
Since $\mathcal{D}$ is independent of $s$, repeated differentiation gives
\begin{equation}\label{eq:Tn-proof}
\frac{d^n}{ds^n}e_m(s)
=
\frac{(-1)^n}{a^n}
(\mathcal{D}^n e)_m.
\end{equation}

To expand $(\mathcal{D}^n e)_m$, choose the shift $+1$ exactly $j$ times, the factor $-b$ exactly $k$ times, and the shift $-1$ with factor $-c$ exactly $\ell$ times, where
$j+k+\ell=n$. The resulting index is $m+j-\ell$, and the multinomial theorem gives
\[
(\mathcal{D}^n e)_m
=
\sum_{\substack{j,k,\ell\geq 0\\j+k+\ell=n}}
\binom{n}{j,k,\ell}
(-1)^{k+\ell}b^k c^\ell
e_{m+j-\ell}(s).
\]
 Substituting into
\eqref{eq:Tn-proof} and using $n=j+k+\ell$, we have
\[
(-1)^n\frac{d^{n}}{ds^{n}}e_m(s)
=
\frac{1}{a^n}
\sum_{\substack{j,k,\ell\ge 0\\ j+k+\ell=n}}
\binom{n}{j,k,\ell}(-1)^{k+\ell}b^k c^\ell\,e_{m+j-\ell}(s),
\]
which is equivalent to the stated formula.
\end{proof}

The first two cases of Theorem~\ref{teor:Derivativ} are
\begin{align*}
\frac{d}{ds}e_m(s)
&=
\frac{1}{a}
\Bigl(
-e_{m+1}(s)+b\,e_m(s)+c\,e_{m-1}(s)
\Bigr),
\quad m\geq 1,
\\
\frac{d^2}{ds^2}e_m(s)
&=
\frac{1}{a^2}
\Bigl(
e_{m+2}(s)-2b\,e_{m+1}(s)
+(b^2-2c)e_m(s)
+2bc\,e_{m-1}(s)
+c^2e_{m-2}(s)
\Bigr),
\quad m\geq 2.
\end{align*}
When $b=0$, the summation index $k$ must be zero, and
Theorem~\ref{teor:Derivativ} becomes
\[
\frac{d^n e_m(s)}{ds^n}
=
\frac{1}{a^n}
\sum_{\ell=0}^{n}
\binom{n}{\ell}
(-1)^{n-\ell}c^\ell
e_{m+n-2\ell}(s),
\qquad m\geq n.
\]
In particular, for the specialization $a=c=1$ and $b=0$,
\[
\frac{d^{n}e_{m}(s)}{ds^{n}}
=\sum_{\ell=0}^{n}\binom{n}{\ell}(-1)^{n-\ell}\,e_{\,m+n-2\ell}(s), \qquad m\geq n.
\]

\section{Combinatorial Interpretation}

In this section, we give a combinatorial interpretation of the specialized sequence $e_n:=e_n(1)$. Although the Laplace parameter is fixed at $s=1$, the sequence $(e_n)_{n\geq 0}$ still depends on the parameters $a$, $b$, and $c$. We construct a weighted tiling model for this sequence and use it to obtain several of its properties from a combinatorial point of view.

It is well known that the Fibonacci number $F_{n+1}$ counts the tilings of a $1\times n$ board by unit squares and dominoes. Benjamin and Quinn~\cite{bookBenjamin} use this tiling model to provide combinatorial proofs of many Fibonacci identities. We now develop an analogous weighted tiling model for the generalized Fibonacci-type polynomial sequence.

Let $\mathcal{T}_n$ be the set of tilings of a $1\times n$ board by unit squares and dominoes.  A unit square has weight either $at$ or $b$, and a domino has
weight $c$. The \emph{weight} of a tiling $P\in\mathcal{T}_n$, denoted by $\wt(P)$, is the product of the weights of its tiles. Let $T_n(t)$ denote the total weight of all tilings in $\mathcal{T}_n$, that is,
\[
T_n(t):=\sum_{P\in\mathcal{T}_n}\wt(P).
\]

For example,
\[
T_3(t)=b^3+2bc+3ab^2t+2act+3a^2bt^2+a^3t^3,
\]
and Figure~\ref{FigTil} shows all tilings contributing to $T_3(t)$.

\begin{figure}[ht!]
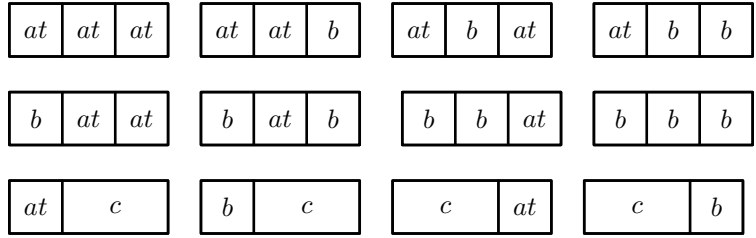

\centering
\noindent
\Tiling[3]{A,A,A}\hspace{0.4cm}\Tiling[3]{A,A,B}\hspace{0.4cm}\Tiling[3]{A,B,A}\hspace{0.4cm} \Tiling[3]{A,B,B}\\[4mm]
\Tiling[3]{B,A,A}\hspace{0.4cm}\Tiling[3]{B,A,B} \hspace{0.4cm}\Tiling[3]{B,B,A}\hspace{0.4cm}\Tiling[3]{B,B,B}\\[4mm] 
\Tiling[3]{A,C}\hspace{0.4cm}\Tiling[3]{B,C}\hspace{0.4cm}\Tiling[3]{C,A}\hspace{0.4cm}\Tiling[3]{C,B}
\caption{All weighted tilings of a $1\times 3$ board contributing to $T_3(t)$.}
\label{FigTil}
\end{figure}

\begin{proposition}\label{prop:Fn-tilings} For every $n\geq 0$, the total weight $T_n(t)$ of all tilings of a $1\times n$ board satisfies \[ T_n(t)=\Ft{n+1}(t). \] Thus, the generalized Fibonacci-type polynomial $\Ft{n+1}(t)$ enumerates these tilings by total weight. \end{proposition}

\begin{proof}
Any tiling of length $n$ ends either with a unit square or with a domino. In the first case, the last tile contributes a factor $at+b$, and the preceding tiles form an arbitrary tiling of length $n-1$. In the second case, the last tile contributes a factor $c$, and the preceding tiles form an arbitrary tiling of length $n-2$. Therefore, for $n\ge 2$,
\[
T_n(t)=(at+b)\,T_{n-1}(t)+c\,T_{n-2}(t).
\]
The initial values are $T_0(t)=1$ (the empty tiling) and $T_1(t)=at+b$. Thus $T_0(t)=\Ft{1}(t)$ and $T_1(t)=\Ft{2}(t)$. Since the sequences $(T_n(t))_{n\ge 0}$ and $(\Ft{n+1}(t))_{n\ge 0}$ satisfy the same recurrence and have the same initial values, they are equal  for all $n\ge 0$.
\end{proof}

\subsection{Combinatorial interpretation of the FiboLaplace sequence}

We now introduce a weighted model based on colored tilings of a $1\times n$ board by unit squares and dominoes. Related colored tiling models have been considered, for example, by Khadir et al.~\cite{KNS}.

There are two types of unit squares. A square of type $A$ has weight $a$, while a square of type $B$ has weight $b$. If a tiling contains exactly $k$ type-$A$ squares, then those squares are colored bijectively with the colors $1,\ldots,k$. Thus, the colors may be assigned in exactly $k!$ ways. Each domino has weight $c$. The weight of a colored tiling is the product of the weights of its tiles. We denote by $\mathcal{C}_n$ the set of all such colored tilings of a $1\times n$ board.

Let $w_n$ be the total weight of all tilings in $\mathcal{C}_n$, that is,
\[
w_n:=\sum_{P\in\mathcal{C}_n}\wt(P).
\]

For example, 
\[
w_3=6a^3+6a^2b+3ab^2+b^3+2ac+2bc,
\]
and Figure~\ref{FigTil2} shows all colored tilings contributing to $w_3$.

\begin{figure}[ht!]
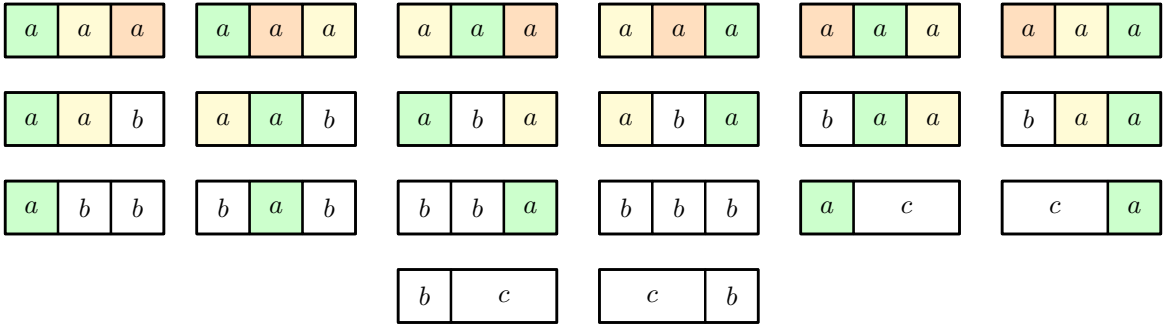

\centering
\noindent
\TilinC[3]{A[1],A[2],A[3]}\hspace{0.4cm}\TilinC[3]{A[1],A[3],A[2]} \hspace{0.4cm}\TilinC[3]{A[2],A[1],A[3]}  \hspace{0.4cm}\TilinC[3]{A[2],A[3],A[1]}   \hspace{0.4cm}\TilinC[3]{A[3],A[1],A[2]}  
\hspace{0.4cm}\TilinC[3]{A[3],A[2],A[1]} \\[4mm]
\TilinC[3]{A[1],A[2],B}\hspace{0.4cm}\TilinC[3]{A[2],A[1],B} \hspace{0.4cm}\TilinC[3]{A[1],B,A[2]}  \hspace{0.4cm}\TilinC[3]{A[2],B,A[1]}   \hspace{0.4cm}\TilinC[3]{B,A[1],A[2]}  
\hspace{0.4cm}\TilinC[3]{B,A[2],A[1]} \\[4mm]
\TilinC[3]{A[1],B,B}\hspace{0.4cm}\TilinC[3]{B,A[1],B} \hspace{0.4cm}\TilinC[3]{B,B,A[1]}  \hspace{0.4cm}\TilinC[3]{B,B,B}   \hspace{0.4cm}\TilinC[3]{A[1],C}  
\hspace{0.4cm}\TilinC[3]{C,A[1]}  \\[4mm]
\TilinC[3]{B,C}  
\hspace{0.4cm}\TilinC[3]{C,B} 
\caption{All colored tilings of a $1\times 3$ board contributing to $w_3$.}
\label{FigTil2}
\end{figure}

Consider the  generating function
\[
E_{a,b,c}(t):=\sum_{n\ge 0} w_nt^n.
\]
Notice that the variable $t$ records the length of the tiling.

We now derive an explicit expression for $E_{a,b,c}(t)$. A colored tiling containing exactly $k$ type-$A$ squares can be uniquely written as
\[
G_0\,\square_{\sigma(1)}\,G_1\,\square_{\sigma(2)}\cdots \square_{\sigma(k)}\,G_k,
\]
where $\sigma$ is a permutation of $\{1,\dots,k\}$ recording,  from left to right, the colors assigned to the type-$A$ squares, and each $G_j$ is a possibly empty tiling consisting  only of type-$B$ squares and dominoes. The generating function for such a block $G_j$ is
\[
G(t)=\sum_{n\geq 0}(bt+ct^2)^n=\frac{1}{1-bt-ct^2}.
\]
For fixed $k$, the colors of the type-$A$ squares can be assigned in $k!$ ways. The $k$ type-$A$ squares contribute $(at)^k$, while the $k+1$ blocks contribute $G(t)^{k+1}$. Therefore,
\[
E_{a,b,c}(t)
=
\sum_{k\geq 0} k!\,(at)^k G(t)^{k+1}
=
G(t)\sum_{k\geq 0} k!\,\bigl(atG(t)\bigr)^k.
\]

Define the formal power series
\[
\Phi(t):=\sum_{k\geq 0}k!\,t^k.
\]
The preceding argument gives the following result.

\begin{theorem}\label{thm:Eabc}
We have 
\[
E_{a,b,c}(t)
=
G(t)\,\Phi\bigl(atG(t)\bigr),
\]
where
\[
G(t)=\frac{1}{1-bt-ct^2}.
\]
\end{theorem}

The following result gives a combinatorial derivation of the recurrence for the sequence $w_n$.

\begin{theorem}\label{thm:wn-recurrence}
The sequence $(w_n)_{n\ge0}$ satisfies, for $n\ge 4$,
\begin{equation}\label{eq:wn-rec}
w_n=(b+an)\,w_{n-1}-ab(n-1)\,w_{n-2}+\bigl(bc-ac(n-2)\bigr)\,w_{n-3}+c^2\,w_{n-4},
\end{equation}
with initial values
\[
w_0=1,\quad 
w_1=a+b,\quad
w_2=2a^2+2ab+b^2+c,\quad
w_3=6a^3+6a^2b+3ab^2+b^3+2ac+2bc.
\]
\end{theorem}
\begin{proof}
Let $\mathcal{Q}_n$ be the set of tilings in $\mathcal{C}_n$ with no type-$A$ squares, and let $q_n$ be their total weight. Such
tilings contain only type-$B$ squares and dominoes. Considering
the last tile gives
\begin{equation}\label{eq:gn-recurrence}
q_n=b\,q_{n-1}+c\,q_{n-2},
\qquad n\geq 2.
\end{equation}
Therefore $w_n-q_n$ is the total weight of the tilings in $\mathcal{C}_n$
containing at least one type-$A$ square. 

We first prove
\begin{equation}\label{eq:wn-positive-recurrence}
an\,w_{n-1}
=
w_n-q_n+c\bigl(w_{n-2}-q_{n-2}\bigr),
\qquad n\geq 2.
\end{equation}
Consider pairs $(P,\gamma)$, where $P\in\mathcal{C}_{n-1}$ and
$\gamma$ is one of the $n$ vertical grid lines of the board, including the two boundary lines. We assign weight $a\wt(P)$ to each such pair. Since there are $n$ choices for $\gamma$, the total weight of all
these pairs is $an\,w_{n-1}$.

Suppose that $P$ contains exactly $k$ type-$A$ squares. If $\gamma$ lies between two tiles or is a boundary line, insert a type-$A$ square
at $\gamma$ and give it color $k+1$.   This gives a tiling in $\mathcal{C}_n$ containing at least one type-$A$ square, and the
construction is weight preserving.

If $\gamma$ passes through a domino, replace that domino by a type-$A$
square of color $k+1$. This gives a tiling $F\in\mathcal{C}_{n-2}$
containing at least one type-$A$ square. Moreover, $a\,\wt(P)=c\,\wt(F)$, since a domino of weight $c$ has been replaced by a type-$A$ square
of weight $a$. 

Both constructions are reversible. In the first case, remove the
type-$A$ square with largest color and mark the grid line created by
its removal. In the second case, replace the type-$A$ square with
largest color by a domino and mark its central grid line.

Thus the first case contributes $w_n-q_n$, while the second contributes
$c\bigl(w_{n-2}-q_{n-2}\bigr)$. Since the two cases partition all
pairs $(P,\gamma)$, \eqref{eq:wn-positive-recurrence} follows.

Now define
\[
r_n:=w_n+c\,w_{n-2}-an\,w_{n-1}.
\]
By~\eqref{eq:wn-positive-recurrence}, $r_n=q_n+c\,q_{n-2}$. Using \eqref{eq:gn-recurrence}, for $n\geq 4$, we have
\begin{align*}
r_n
&=q_n+c\,q_{n-2}\\
&=b\,q_{n-1}+c\,q_{n-2}
  +bc\,q_{n-3}+c^2q_{n-4}\\
&=b\bigl(q_{n-1}+c\,q_{n-3}\bigr)
  +c\bigl(q_{n-2}+c\,q_{n-4}\bigr)\\
&=b\,r_{n-1}+c\,r_{n-2}.
\end{align*}
Substituting the definition of $r_n$ gives
\begin{multline*}
w_n+c\,w_{n-2}-an\,w_{n-1}
\\=
b\Bigl(
w_{n-1}+c\,w_{n-3}-a(n-1)w_{n-2}
\Bigr)+
c\Bigl(
w_{n-2}+c\,w_{n-4}-a(n-2)w_{n-3}
\Bigr).
\end{multline*}
Canceling $c\,w_{n-2}$ and collecting terms proves \eqref{eq:wn-rec}. The initial values follow by direct inspection of the tilings of lengths $0,1,2$, and $3$.
\end{proof}
\par\medskip

Theorems~\ref{thm:wn-recurrence} and~\ref{thm:en-4term} show that,
after specializing $s=1$, the sequences $(w_n)_{n\geq 0}$ and
$(e_{n+1})_{n\geq 0}$ satisfy the same fourth-order recurrence and
have the same initial values. They therefore coincide. This gives the
following weighted tiling interpretation of the FiboLaplace sequence.
\begin{theorem}\label{thm:en=wn}
Under the specialization $s=1$, for every $n\geq 0$,
\[
e_{n+1}=w_n.
\]
Equivalently, the FiboLaplace number $e_{n+1}$ is the total weight of
the tilings in $\mathcal{C}_n$.
\end{theorem}

The tiling model in Theorem~\ref{thm:en=wn} also yields an explicit
formula for the FiboLaplace sequence. We refine the total weight according to the number of dominoes and type-$A$ squares.

\begin{theorem}\label{en-cerrada1}
For all $n\geq 1$,
\[
e_n=\sum_{i=0}^{\lfloor\frac{n-1}{2}\rfloor}\ \sum_{k=0}^{n-2i-1}
\binom{n-i-1}{i}\binom{n-2i-1}{k}\,a^{k}\,b^{n-2i-1-k}\,c^{i}\,k!.
\]
\end{theorem}

\begin{proof}
By Theorem~\ref{thm:en=wn}, we have $e_n=w_{n-1}$. Hence $e_n$ is
the total weight of the colored tilings of a $1\times(n-1)$ board. Fix $i$ and $k$, and consider tilings having exactly $i$ dominoes and
$k$ type-$A$ squares. The dominoes cover $2i$ cells, leaving
$n-2i-1$ cells to be covered by squares. Regarding each domino as a
single object, such a tiling consists of $i$ dominoes and
$n-2i-1$ squares. Thus the dominoes can be placed in $\binom{n-i-1}{i}$ ways.

Next, choose  $k$ of these $n-2i-1$ squares to be of type $A$, the
remaining $n-2i-1-k$ squares are of type $B$. There are $\binom{n-2i-1}{k}$ ways to make this choice.  The $k$ type-$A$ squares can then be assigned the colors $1,\ldots,k$ in $k!$ ways. Therefore, the total contribution of the tilings with $i$ dominoes
and $k$ type-$A$ squares is 
\[
\binom{n-i-1}{i}\binom{n-2i-1}{k}\,a^{k}\,b^{n-2i-1-k}\,c^{i}\,k!.
\]
Finally, summing over  $0\leq i\leq \left\lfloor\frac{n-1}{2}\right\rfloor$ and $0\leq k\leq n-2i-1$ gives the result.
\end{proof}

\begin{remark}
The preceding formula can also be obtained analytically. Indeed, the explicit expansion
\begin{equation*}
\Ft{n}(t)
=
\sum_{i=0}^{\left\lfloor (n-1)/2\right\rfloor}
\binom{n-i-1}{i}(at+b)^{n-2i-1}c^i,
\qquad n\geq 1,
\end{equation*}
is well known; see, for example, \cite{FlorezJC}. Expanding
$(at+b)^{n-2i-1}$ by the binomial theorem and using $
\mathcal{L}\{t^k\}(s)=k!/s^{k+1}$,  we obtain
\[
e_n(s)
=
\sum_{i=0}^{\left\lfloor (n-1)/2\right\rfloor}
\sum_{k=0}^{n-2i-1}
\binom{n-i-1}{i}
\binom{n-2i-1}{k}
a^k b^{n-2i-1-k}c^i
\frac{k!}{s^{k+1}}.
\]
Setting $s=1$ recovers the formula in
Theorem~\ref{en-cerrada1}. When $b=0$, only the term $k=n-2i-1$ remains, and hence
\begin{equation*}
e_n(s)
=
\sum_{i=0}^{\left\lfloor (n-1)/2\right\rfloor}
\binom{n-i-1}{i}
c^i
\frac{(n-2i-1)!a^{n-2i-1}}{s^{n-2i}}.
\end{equation*}
\end{remark}

We next ignore the colors and keep track only of the number of type-$A$ squares. The resulting polynomials will be used to derive an addition formula for the FiboLaplace sequence in the specialization $s=1$.

Let $u_{n,k}$ be the total weight of the uncolored tilings of a
$1\times n$ board containing exactly $k$ type-$A$ squares, where
type-$A$ squares, type-$B$ squares, and dominoes have weights $a$,
$b$, and $c$, respectively.

Define
\[
U_n(y):=\sum_{k\ge 0} u_{n,k} y^k,
\]
so that $y$ records the number of type-$A$ squares. For example, forgetting the colors in Figure~\ref{FigTil2} gives
\[
U_3(y)=a^3y^3 + 3a^2by^2 + (3ab^2+2ac)y + b^3 + 2bc.
\]

\begin{theorem}\label{teor:addition-formula}
For all $n,m\ge 1$,
\begin{equation*}
U_{n+m}(y)=U_n(y)\,U_m(y)+c\,U_{n-1}(y)\,U_{m-1}(y).
\end{equation*}
Moreover, after setting $s=1$,
\[
e_{n+m+1}
=
\sum_{i,j\geq 0}
(i+j)!\,u_{n,i}u_{m,j}
+
c\sum_{i,j\geq 0}
(i+j)!\,u_{n-1,i}u_{m-1,j}.
\]
\end{theorem}

\begin{proof}
Cut a $1\times(n+m)$ board between cells $n$ and $n+1$. If no domino crosses the cut, the tiling is uniquely obtained by concatenating a tiling of length $n$ with one of length $m$. These
tilings have total weight  $U_n(y)U_m(y)$. 

If a domino crosses the cut, remove it. The remaining cells form a
board of length $n-1$ to the left of the cut and a board of length
$m-1$ to the right. Since the removed domino has weight $c$, these
tilings have total weight $c\,U_{n-1}(y)U_{m-1}(y)$. The two cases prove the first identity.

We use the same decomposition for colored tilings. In the first case,
suppose that the left and right tilings contain $i$ and $j$ type-$A$ squares, respectively. Their uncolored total weight is $u_{n,i}u_{m,j}$. After concatenation, the resulting tiling has $i+j$ type-$A$ squares, which can be colored in $(i+j)!$ ways.
Thus this case contributes
\[
\sum_{i,j\geq 0}(i+j)!\,u_{n,i}u_{m,j}.
\]
If a domino crosses the cut, the same argument gives
\[
c\sum_{i,j\geq 0}
(i+j)!\,u_{n-1,i}u_{m-1,j}.
\]
Hence the right-hand side is the total weight $w_{n+m}$ of the colored
tilings of a $1\times(n+m)$ board.  By Theorem~\ref{thm:en=wn}, after setting $s=1$,
$w_{n+m}=e_{n+m+1}$, which proves the second identity.
\end{proof}

The first identity is a weighted version of the classical Fibonacci addition formula. Indeed, since a unit square has total weight $ay+b$, specializing  $ay+b=1$ and $c=1$ gives $U_n(y)=F_{n+1}$, and the
identity becomes
\[
F_{n+m+1}=F_{n+1}F_{m+1}+F_nF_m.
\]
The second identity is its colored analogue for the FiboLaplace sequence.

\subsection{The FiboLaplace triangle}

We now refine the tiling interpretation of the FiboLaplace sequence by recording the number of dominoes. This refinement leads naturally to a triangular array whose row sums recover the FiboLaplace sequence.

For $n\geq 0$ and $0\leq k\leq \lfloor n/2\rfloor$, let $\mathcal C_n^{(k)}$ denote the set of tilings in $\mathcal C_n$ with
exactly $k$ dominoes, and define
\[
h(n,k)=\sum_{T\in\mathcal C_n^{(k)}}\wt(T).
\]
Thus $h(n,k)$ is the total weight of the colored tilings of a $1\times n$ board having exactly $k$ dominoes. For example, the tilings
in Figure~\ref{FigTil2} give  $
h(3,1)=2ac+2bc$.

The numbers $h(n,k)$ form a triangular array, which we call the
\emph{FiboLaplace triangle}. By Theorem~\ref{thm:en=wn},
\[
e_{n+1}=w_n=\sum_{k=0}^{\lfloor n/2\rfloor}h(n,k).
\]
Thus, the row sums of the FiboLaplace triangle recover the FiboLaplace sequence.

\begin{theorem}\label{thm:h-n-k}
Let $n\geq 0$ and $0\leq k\leq \lfloor n/2\rfloor$. Then
\[
h(n,k)
=
\binom{n-k}{k}c^k
\sum_{r=0}^{n-2k}
\binom{n-2k}{r}r!a^r b^{n-2k-r}.
\]
In particular, $h(0,0)=1$.
\end{theorem}
\begin{proof}
Consider a tiling of a $1\times n$ board with $k$ dominoes and  $m=n-2k$  unit squares. Since each domino covers two cells, the tiling
has $n-k$ tiles in all. Collapsing each domino to a single tile, we
obtain a sequence of $n-k$ tiles, exactly $k$ of which are dominoes.
Thus, there are  $\binom{n-k}{k}$ possible underlying tilings.

Now fix one such tiling and let $r$ be the number of unit squares of type $A$.  There are $\binom{m}{r}$ ways to choose these squares, while
the remaining $m-r$ squares are of type $B$. Their combined weight is 
$a^r b^{m-r}$.  The $r$ squares of type $A$ can be assigned distinct colors in $r!$
ways, and the $k$ dominoes contribute the factor $c^k$. Hence the total weight of the tilings with exactly $r$ squares of type $A$ is
\[
\binom{n-k}{k}\binom{m}{r}c^k a^r b^{m-r} r!.
\]
Summing over all possible values of $r$ gives
\[
h(n,k)
=\binom{n-k}{k}\,c^k\sum_{r=0}^{m}\binom{m}{r}r!a^r b^{m-r}.
\]
Substituting $m=n-2k$ proves the result. 
\end{proof}

The first few rows of the FiboLaplace triangle are
\tiny
\[
\begin{array}{c|ccccc}
n\backslash k & 0 & 1 & 2 & 3 \\
\hline
0 & 1 \\
1 & a+b \\
2 & 2a^2+2ab+b^2 & c \\
3 & 6a^3+6a^2b+3ab^2+b^3 & 2c(a+b) \\
4 & 24a^4+24a^3b+12a^2b^2+4ab^3+b^4 & 3c(2a^2+2ab+b^2) & c^2 \\
5 & 120a^5+120a^4b+60a^3b^2+20a^2b^3+5ab^4+b^5
  & 4c(6a^3+6a^2b+3ab^2+b^3) & 3c^2(a+b) \\
6 & 720a^6+720a^5b+360a^4b^2+120a^3b^3+30a^2b^4+6ab^5+b^6
  & 5c(24a^4+24a^3b+12a^2b^2+4ab^3+b^4)
  & 6c^2(2a^2+2ab+b^2) & c^3
\end{array}
\]
\normalsize

Several boundary entries follow immediately from
Theorem~\ref{thm:h-n-k}. For every $n\geq 0$,
\[
h(n,0)
=
\sum_{r=0}^{n}\binom{n}{r}r!a^r b^{n-r},
\]
while, for every $k\geq 0$, $h(2k,k)=c^k$
 and $h(2k+1,k)=(k+1)c^k(a+b)$.

When $b=0$, the sum in Theorem~\ref{thm:h-n-k} reduces to a single
term, giving the following simple formula.

\begin{corollary}\label{cor:b0}
Let $n\geq 0$ and $0\leq k\leq \lfloor n/2\rfloor$. If $b=0$, then
\[h(n,k)
=
\binom{n-k}{k}c^k(n-2k)!a^{n-2k}.\]
\end{corollary}

The explicit formula above gives each entry of the FiboLaplace triangle directly. The tiling model also yields a recurrence that constructs the triangle row by row. 

\begin{theorem}\label{thm:rec-h}
For $n\geq 1$ and $0\leq k\leq \lfloor n/2\rfloor$, the sequence $h(n,k)$ satisfies
\begin{equation*}
h(n,k)
=(a+b)\,h(n-1,k)+a^{2}\,\frac{\partial}{\partial a}h(n-1,k)+c\,h(n-2,k-1),
\end{equation*}
with initial conditions $h(0,0)=1$ and $
h(n,k)=0$ whenever $n<0$, $k<0$, or $2k>n$.
\end{theorem}

\begin{proof}
We partition the tilings in $\mathcal C_n^{(k)}$ according to their
last tile. If the last tile is a square of type $B$, then deleting it leaves a tiling in $\mathcal C_{n-1}^{(k)}$. This case contributes $b\,h(n-1,k)$. If the last tile is a domino, then deleting it leaves a tiling in $\mathcal C_{n-2}^{(k-1)}$ and the contribution is
$c\,h(n-2,k-1)$.

It remains to consider tilings ending in a square of type $A$.  Let $T'\in \mathcal C_{n-1}^{(k)}$, and suppose that $T'$ has $r=r(T')$ squares of type $A$. There are $r+1$ ways to extend $T'$ by appending
and coloring a new square of type $A$.  Since the new square contributes a
factor of $a$, the total weight of these extensions is  $a(r+1)\,\wt(T')$.

Summing over all $T'\in \mathcal C_{n-1}^{(k)}$, we obtain
\[
a\sum_{T'}(r(T')+1)\wt(T')
=
a\,h(n-1,k)
+a\sum_{T'}r(T')\wt(T').
\]
Since differentiation with respect to $a$ marks the number of squares
of type $A$, we have
\[
a\frac{\partial}{\partial a}h(n-1,k)
=
\sum_{T'}r(T')\wt(T').
\]
Hence
\[
a\sum_{T'} r(T')\,\wt(T')
=
a^2\frac{\partial}{\partial a}h(n-1,k).
\]
Thus the contribution of tilings ending in a type-$A$ square is
\[
a\,h(n-1,k)+a^2\frac{\partial}{\partial a}h(n-1,k).
\]
Adding the contributions from the three possible last tiles gives the desired result. The boundary conditions follow directly from the definition.
\end{proof}

\section{Continued-Fraction Expressions} This section gives two continued-fraction expansions for the generating function $E_{a,b,c}(t)$ of the FiboLaplace sequence $(e_{n+1})_{n\geq 0}$. We begin with Euler's classical continued fraction for the factorial series; see, for example, \cite{Sokal}. As an identity of formal power series,
\begin{equation}\label{fcEuler}
\Phi(t)=\sum_{k\geq 0}k!\,t^k=\cfrac{1}{1-\cfrac{1\cdot t}{1-\cfrac{1\cdot t}{1-\cfrac{2\cdot t}{1-\cfrac{2\cdot t}{1-\cfrac{3\cdot t}{1-\cfrac{3\cdot t}{\ddots}}}}}}}.
\end{equation}

\begin{theorem}\label{thm:continued-fraction}
The generating function of the FiboLaplace sequence has the
continued-fraction expansion
\begin{equation*}
E_{a,b,c}(t)
=
\sum_{n\geq 0}e_{n+1}t^n
=
\cfrac{1}{
1-bt-ct^2-
\cfrac{at}{
1-\cfrac{at}{
1-bt-ct^2-
\cfrac{2at}{
1-\cfrac{2at}{
1-bt-ct^2-
\cfrac{3at}{
1-\cfrac{3at}{\ddots}}}}}}}.
\end{equation*}
\end{theorem}

\begin{proof}
By Theorem~\ref{thm:Eabc},  $E_{a,b,c}(t) =
G(t)\,\Phi\bigl(atG(t)\bigr)$, where
\[
G(t)=\frac{1}{1-bt-ct^2}.
\]
Set $A(t)=1-bt-ct^2$, so that $G(t)=1/A(t)$. Substituting $t$ by $\frac{at}{A(t)}$ into \eqref{fcEuler} gives
\[
E_{a,b,c}(t)
=
\frac{1}{A(t)}
\cfrac{1}{
1-\cfrac{\frac{at}{A(t)}}{
1-\cfrac{\frac{at}{A(t)}}{
1-\cfrac{\frac{2at}{A(t)}}{
1-\cfrac{\frac{2at}{A(t)}}{
1-\cfrac{\frac{3at}{A(t)}}{
1-\cfrac{\frac{3at}{A(t)}}{\ddots}}}}}}}.
\]
Absorbing the outer factor $1/A(t)$ into the first denominator and
then clearing the factors $A(t)$ recursively at every other level
yields  the desired continued fraction.
\end{proof}

The continued fraction in Theorem~\ref{thm:continued-fraction} can be
contracted to a Jacobi-type continued fraction. Indeed, contracting
Euler's continued fraction in \eqref{fcEuler} gives the classical
expansion
\[
\Phi(t)=\sum_{k\geq 0}k!t^k
=
\cfrac{1}{
1-t-
\cfrac{t^2}{
1-3t-
\cfrac{4t^2}{
1-5t-
\cfrac{9t^2}{
1-7t-\ddots
}
}
}
};
\]
see, for example, \cite{Flajolet}. This yields the following
Jacobi-type continued fraction for the FiboLaplace generating
function.

\begin{theorem}\label{thm:contracted-cf}
The generating function $E_{a,b,c}(t)$ admits
the continued fraction expansion
\[
\cfrac{1}{
1-(a+b)t-ct^2-
\cfrac{a^2t^2}{
1-(3a+b)t-ct^2-
\cfrac{4a^2t^2}{
1-(5a+b)t-ct^2-
\cfrac{9a^2t^2}{
1-(7a+b)t-ct^2-\ddots
}
}
}
}.
\]
\end{theorem}

We next give a lattice-path interpretation of
Theorem~\ref{thm:contracted-cf}. Let $\mathcal{GM}_n$ denote the set of
paths from $(0,0)$ to $(n,0)$ that never pass below the horizontal axis
and use the steps
\[
U=(1,1),\qquad D=(1,-1),\qquad H=(1,0),\qquad H_2=(2,0).
\]
Thus, $\mathcal{GM}_n$ consists of Motzkin paths whose horizontal steps
may have length one or two. We assign weight $1$ to each up-step, weight
$(j+1)^2a^2$ to a down-step from height $j+1$ to height $j$, weight
$b+(2j+1)a$ to a step $H$ at height $j$, and weight $c$ to each step
$H_2$. The weight $\wt(P)$ of a path $P$ is the product of the weights
of its steps. For lattice-path interpretations of the Fibonacci numbers
and related generalizations, see \cite{CheonKimShapiro,RamirezSirventkBonacci}.

\begin{theorem}\label{thm:motzkin-interpretation}
For every $n\geq 0$,
\[
e_{n+1}
=
\sum_{P\in\mathcal{GM}_n}\wt(P).
\]
Equivalently, the FiboLaplace number $e_{n+1}$ is the total weight of
the generalized Motzkin paths of  length $n$.
\end{theorem}

\begin{proof}
For $j\geq 0$, let $M_j(t)$ be the generating function for generalized
Motzkin paths that begin and end at height $j$ and never pass below
that height, where $t$ records length.

A nonempty path
counted by $M_j(t)$ begins with either a step $H$, a step $H_2$, or an
up-step. In the first case, the initial step has weight
$b+(2j+1)a$ and is followed by another path counted by $M_j(t)$. In
the second case, the initial step has weight $c$ and  length
two, and is again followed by a path counted by $M_j(t)$. In the last
case, the up-step is followed by a path counted by $M_{j+1}(t)$, a
down-step from height $j+1$ to height $j$ of weight $(j+1)^2a^2$, and
then another path counted by $M_j(t)$.

Therefore,
\[
M_j(t)
=
1+\bigl(b+(2j+1)a\bigr)tM_j(t)
+ct^2M_j(t)
+(j+1)^2a^2t^2M_{j+1}(t)M_j(t).
\]
Solving for $M_j(t)$ gives
\[
M_j(t)
=
\frac{1}{
1-\bigl(b+(2j+1)a\bigr)t-ct^2
-(j+1)^2a^2t^2M_{j+1}(t)
}.
\]
Iterating this identity, beginning with $j=0$, yields
\[
M_0(t)
=
\cfrac{1}{
1-(a+b)t-ct^2-
\cfrac{a^2t^2}{
1-(3a+b)t-ct^2-
\cfrac{4a^2t^2}{
1-(5a+b)t-ct^2-
\cfrac{9a^2t^2}{
1-(7a+b)t-ct^2-\ddots
}
}
}
}.
\]
By Theorem~\ref{thm:contracted-cf}, this continued fraction is
$E_{a,b,c}(t)$. Hence
\[
M_0(t)
=
E_{a,b,c}(t)
=
\sum_{n\geq 0}e_{n+1}t^n.
\]
Comparing coefficients of $t^n$ completes the proof.
\end{proof}

For example, for $n=3$, the paths in $\mathcal{GM}_3$ are shown in Figure~\ref{exMotzkin}. 

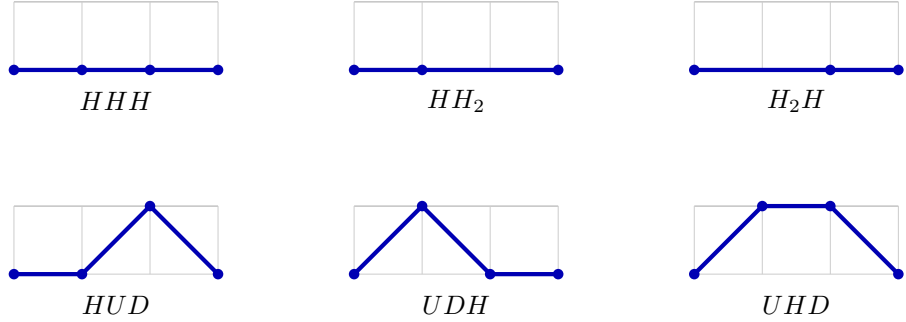
\begin{figure}[ht!]
\centering
\begin{tikzpicture}[
    x=0.9cm, y=0.9cm,
    grid/.style={draw=gray!40, line width=0.35pt},
    level/.style={draw=gray!70, line width=0.35pt},
    path/.style={blue!70!black, line width=1.6pt, line cap=round, line join=round},
    vertex/.style={circle, fill=blue!70!black, inner sep=1.4pt},
    lab/.style={font=\small}
]


\begin{scope}[shift={(0,0)}]
    \draw[grid] (0,0) -- (3,0);
    \draw[level] (0,1) -- (3,1);
    \foreach \x in {0,1,2,3}
        \draw[grid] (\x,0) -- (\x,1);
    \draw[path] (0,0) -- (1,0) -- (2,0) -- (3,0);
    \foreach \P in {(0,0),(1,0),(2,0),(3,0)}
        \node[vertex] at \P {};
    \node[lab] at (1.5,-0.45) {$HHH$};
\end{scope}

\begin{scope}[shift={(5,0)}]
    \draw[grid] (0,0) -- (3,0);
    \draw[level] (0,1) -- (3,1);
    \foreach \x in {0,1,2,3}
        \draw[grid] (\x,0) -- (\x,1);
    \draw[path] (0,0) -- (1,0) -- (3,0);
    \foreach \P in {(0,0),(1,0),(3,0)}
        \node[vertex] at \P {};
    \node[lab] at (1.5,-0.45) {$HH_2$};
\end{scope}

\begin{scope}[shift={(10,0)}]
    \draw[grid] (0,0) -- (3,0);
    \draw[level] (0,1) -- (3,1);
    \foreach \x in {0,1,2,3}
        \draw[grid] (\x,0) -- (\x,1);
    \draw[path] (0,0) -- (2,0) -- (3,0);
    \foreach \P in {(0,0),(2,0),(3,0)}
        \node[vertex] at \P {};
    \node[lab] at (1.5,-0.45) {$H_2H$};
\end{scope}


\begin{scope}[shift={(0,-3)}]
    \draw[grid] (0,0) -- (3,0);
    \draw[level] (0,1) -- (3,1);
    \foreach \x in {0,1,2,3}
        \draw[grid] (\x,0) -- (\x,1);
    \draw[path] (0,0) -- (1,0) -- (2,1) -- (3,0);
    \foreach \P in {(0,0),(1,0),(2,1),(3,0)}
        \node[vertex] at \P {};
    \node[lab] at (1.5,-0.45) {$HUD$};
\end{scope}

\begin{scope}[shift={(5,-3)}]
    \draw[grid] (0,0) -- (3,0);
    \draw[level] (0,1) -- (3,1);
    \foreach \x in {0,1,2,3}
        \draw[grid] (\x,0) -- (\x,1);
    \draw[path] (0,0) -- (1,1) -- (2,0) -- (3,0);
    \foreach \P in {(0,0),(1,1),(2,0),(3,0)}
        \node[vertex] at \P {};
    \node[lab] at (1.5,-0.45) {$UDH$};
\end{scope}

\begin{scope}[shift={(10,-3)}]
    \draw[grid] (0,0) -- (3,0);
    \draw[level] (0,1) -- (3,1);
    \foreach \x in {0,1,2,3}
        \draw[grid] (\x,0) -- (\x,1);
    \draw[path] (0,0) -- (1,1) -- (2,1) -- (3,0);
    \foreach \P in {(0,0),(1,1),(2,1),(3,0)}
        \node[vertex] at \P {};
    \node[lab] at (1.5,-0.45) {$UHD$};
\end{scope}

\end{tikzpicture}
\caption{The six paths in $\mathcal{GM}_3$.}
\label{exMotzkin}
\end{figure}
Their respective weights are
\[
(a+b)^3,\qquad c(a+b),\qquad c(a+b),\qquad
a^2(a+b),\qquad a^2(a+b),\qquad a^2(3a+b).
\]
Hence
\[
\begin{aligned}
\sum_{P\in\mathcal{GM}_3}\wt(P)
&=(a+b)^3+2c(a+b)+2a^2(a+b)+a^2(3a+b)\\
&=6a^3+6a^2b+3ab^2+b^3+2ac+2bc\\
&=e_4.
\end{aligned}
\]
This agrees with the value obtained from the colored-tiling
interpretation in Figure~\ref{FigTil2}.

Thus, the FiboLaplace sequence admits two different combinatorial interpretations, one in terms of colored tilings and the other in terms of generalized Motzkin paths. Finding a direct weight-preserving bijection between these two families remains an interesting open problem.

In a forthcoming paper, we will study arithmetic aspects of the FiboLaplace sequence, including divisibility properties, congruences, and $p$-adic
valuations.

\section*{Funding}
The first author was partially supported
by the Citadel Foundation. The second author was supported by UPTC. The
last author was partially supported by Universidad Nacional de Colombia, Project No. 64041. 
\vspace{-0.3cm}
\section*{Contributions}
All authors contributed equally to the manuscript and read and approved the final manuscript.
\vspace{-0.3cm}

\section*{Conflict of interest}
This study does not have any conflicts to disclose.

\end{document}